\documentclass[12pt]{amsart}

\usepackage{amsmath,amsthm,amssymb} 
\usepackage{tikz}
\usetikzlibrary{arrows,positioning,calc}

\usepackage{a4wide}
\usepackage{enumerate}
\usepackage{multicol, multirow, caption}
\newtheorem{theorem}{Theorem}[section]
\newtheorem{lemma}[theorem]{Lemma}
\newtheorem{proposition}[theorem]{Proposition}
\newtheorem{corollary}[theorem]{Corollary}

\theoremstyle{remark}
\newtheorem{remark}[theorem]{Remark}

\newtheorem*{claim*}{Claim}
\newtheorem{fact}{Fact}

\newcommand{\C}{\ensuremath{\mathbb{C}}}
\newcommand{\R}{\ensuremath{\mathbb{R}}}
\newcommand{\K}{\ensuremath{\mathbb{K}}}
\newcommand{\sph}{\ensuremath{\mathbb{S}}}
\newcommand{\g}[1]{\ensuremath{\mathfrak{#1}}}

\newcommand{\cal}[1]{\ensuremath{\mathcal{#1}}}

\DeclareMathOperator{\id}{Id}
\DeclareMathOperator{\Ad}{Ad}
\DeclareMathOperator{\Aut}{Aut}
\DeclareMathOperator{\Out}{Out}

\DeclareMathOperator{\spann}{span}
\DeclareMathOperator{\diag}{diag}

\renewcommand{\H}{\ensuremath{\mathbb{H}}}
\renewcommand{\mod}[1]{\ensuremath{\;(\mathrm{mod\;}#1)}}
\renewcommand{\P}{\ensuremath{\mathsf{P}}}

\newcommand{\GL}[1]{\ensuremath{\mathsf{GL}(#1)}}
\newcommand{\SU}[1]{\ensuremath{\mathsf{SU}(#1)}}
\newcommand{\SUxU}[2]{\ensuremath{\mathsf{S(U}(#1)\times\mathsf{U}(#2))}}
\newcommand{\U}[1]{\ensuremath{\mathsf{U}(#1)}}
\newcommand{\Sp}[1]{\ensuremath{\mathsf{Sp}(#1)}}
\newcommand{\SO}[1]{\ensuremath{\mathsf{SO}(#1)}}
\newcommand{\OG}[1]{\ensuremath{\mathsf{O}(#1)}}
\newcommand{\Spin}[1]{\ensuremath{\mathsf{Spin}(#1)}}

\newcommand{\G}{\ensuremath{\mathsf{G}_2}}
\newcommand{\E}[1]{\ensuremath{\mathsf{E}_{#1}
		}}

\begin{document}
\title[On isoparametric foliations of projective spaces]{On isoparametric foliations\\of complex and quaternionic projective spaces}
\author[M.~Dom\'{\i}nguez-V\'{a}zquez]{Miguel Dom\'{\i}nguez-V\'{a}zquez}
\address{CITMAga, 15782 Santiago de Compostela, Spain.\newline\indent Department of Mathematics, Universidade de Santiago de Compostela, Spain}
\email{miguel.dominguez@usc.es}
\author[A.~Kollross]{Andreas Kollross}
\address{Institut f\"ur Geometrie und Topologie, Universit\"at Stuttgart, Germany.}
\email{andreas.kollross@mathematik.uni-stuttgart.de}

\begin{abstract}
We conclude the classification of isoparametric (or equivalently, polar) foliations of complex and quaternionic projective spaces. This is done by investigating the projections of certain inhomogeneous isoparametric foliations of the $31$-sphere under the respective Hopf fibrations, thereby solving the last remaining open cases.
\end{abstract}

\thanks{The first author has been supported by the grant PID2022-138988NB-I00 funded by MICIU/AEI/10.13039/501100011033 and by ERDF, EU, and by the projects ED431F 2020/04 and ED431C 2023/31 (Xunta de Galicia, Spain).
}

\maketitle

\section{Introduction}
A \emph{polar foliation} $\mathcal{F}$ of a complete Riemannian manifold $M$ is a singular Riemannian foliation of $M$ such that each point of $M$ is contained in a totally geodesic submanifold of $M$ (called section) that intersects all leaves of $\mathcal{F}$ and always perpendicularly. Polar foliations whose regular leaves have parallel mean curvature are called \emph{isoparametric foliations}. A polar foliation of $M$ is called \emph{homogeneous} if it is the family of orbits of an isometric action on~$M$; such an isometric action is said to be \emph{polar}. Every homogeneous polar foliation turns out to be isoparametric.

General results on polar foliations of compact symmetric spaces have been derived by Lytchak~\cite{Lytchak:gafa}, and Liu and Radeschi~\cite{LiuRadeschi}. In particular, the latter proved that polar foliations of symmetric spaces of compact type are isoparametric. Also, a complete classification of homogeneous polar foliations is known on irreducible spaces~\cite{PT:jdg, Ko:jdg, KL:blms, GK:agag}. However, classification results of (not necessarily homogeneous) polar or isoparametric foliations are very scarce. Apart from the classical but far-reaching case of round spheres $\mathbb{S}^n$, to whose solution many authors have contributed, we only have classifications in complex projective spaces $\C \P^n$, $n\neq 15$, and quaternionic projective spaces $\H \P^n$, $n\neq 7$, obtained in~\cite{Do:tams} and~\cite{DG:tohoku}. We refer the reader to the recent survey~\cite{Th:survey22} for more information on polar foliations of symmetric spaces, and to~\cite{Chi} for the particularly important case of codimension one polar foliations (or equivalently, isoparametric families of hypersurfaces) of round spheres.

The aim of this article is to conclude the classification of polar foliations of complex and quaternionic projective spaces. To this end, it suffices to classify polar foliations of codimension one on the exceptional open cases $\C \P^{15}$ and $\H \P^7$ by the works by the first-named author and Gorodski~\cite{Do:tams,DG:tohoku}. But before stating our main result, we need to recall the approach followed in~\cite{Do:tams,DG:tohoku} and the main reason why this approach was insufficient to deal with the exceptional cases. We anticipate that this reason is not solely related to the classification problem on spheres.

It is a well-known fact that a singular Riemannian foliation $\cal{G}$ of $\C \P^n$ (resp.\ on $\H \P^n$) is polar if and only if its pullback foliation $\pi^{-1}\cal{G}$ under the Hopf fibration $\pi\colon \mathbb{S}^{2n+1}\to\C \P^n$ (resp.\ $\pi\colon \mathbb{S}^{4n+3}\to\H \P^n$) is a polar foliation of the round sphere $\mathbb{S}^{2n+1}$ (resp.\ $\mathbb{S}^{4n+3}$). In other words, the polar foliations of a projective space $\C \P^n$ or $\H \P^n$ are nothing but the projections (under the Hopf map) of the polar foliations of the corresponding sphere whose leaves are in turn foliated by the Hopf fibers. However, the following interesting fact arises: \emph{there are non-congruent polar foliations of $\C \P^n$ (or $\H \P^n$) with congruent pullback foliations of $\mathbb{S}^{2n+1}$ (resp.\ $\mathbb{S}^{4n+3}$)}. Roughly speaking, this is due to the fact that a projective space has an  isometry group that is `smaller' than that of the corresponding sphere. This is also the reason for the existence of \emph{inhomogeneous polar foliations of complex and quaternionic projective spaces whose pullback foliations are homogeneous}, and a fundamental element for the striking result stating that all irreducible polar foliations of $\C \P^n$ (or of $\H \P^n$) are homogeneous if and only if $n+1$ is prime, see~\cite{Do:tams,DG:tohoku}.

In order to investigate these remarkable phenomena, the idea is the following: for each polar foliation $\cal{F}$ of a round sphere, analyze all possible congruence classes of polar foliations $\cal{G}$ of $\C \P^n$ (or $\H \P^n$) such that $\pi^{-1}\cal{G}$ is congruent to $\cal{F}$. This
is equivalent to the following problem (see below for definitions):
\begin{enumerate}[\rm (P)]\label{problem}
\item For each polar foliation $\cal{F}$ of $\mathbb{S}^{2n+1}$ (resp.\ $\mathbb{S}^{4n+3}$), determine a maximal set of complex structures $\{J_i\}_{i\in I_\C}$ on $\R^{2n+2}$  (resp.\ of quaternionic structures $\{\g{q}_i\}_{i\in I_\H}$ on $\R^{4n+4}$) preserving $\cal{F}$ and such that the polar foliations $\pi_{J_i}(\cal{F})$, $i\in I_\C$, on $\C \P^n$ (resp.\  $\pi_{\g{q}_i}(\cal{F})$, $i\in I_\H$, on $\H \P^n$) are pairwise non-congruent.
\end{enumerate}
For convenience, we define $N_\C(\cal{F}):=|I_\C|$ or $N_\H(\cal{F}):=|I_\H|$ as the cardinals of the maximal sets mentioned above; a posteriori, these will be finite numbers. In Problem (P) and throughout this paper, by \emph{complex structure} on a Euclidean space $\R^{2n+2}$ we mean a linear isometry $J$ of $\R^{2n+2}$ such that $J^2=-\id$ (equivalently, $J\in\g{so}(2n+2)\cap\SO{2n+2}$). By \emph{quaternionic structure} on a Euclidean space $\R^{4n+4}$ we mean a $3$-dimensional Lie subalgebra $\g{q}\cong\g{su}(2)$ of $\g{so}(4n+4)\subset \mathrm{End}(\R^{4n+4})$ admitting a basis $J_1,J_2,J_3\in\g{q}$ such that $J_iJ_{i+1}=J_{i+2}$ (indices modulo $3$) and $J_i^2=-\id$, for $i=1,2,3$. Notice that a complex structure $J$ on $\R^{2n+2}$ determines a Hopf fibration $\pi_J\colon \mathbb{S}^{2n+1}\to\C \P^n$ from the unit sphere $\mathbb{S}^{2n+1}$ of $\R^{2n+2}$, given by the quotient map of the $\U{1}$-action of the connected subgroup of $\SO{2n+2}$ with Lie algebra $\R J$. Similarly, a quaternionic structure $\g{q}$ on $\R^{4n+4}$ determines a Hopf fibration $\pi_\g{q}\colon \mathbb{S}^{4n+3}\to\H \P^n$ from the unit sphere $\mathbb{S}^{4n+3}$ of $\R^{4n+4}$, induced by the $\SU{2}$-action of the connected subgroup of $\SO{4n+4}$ with Lie algebra $\g{q}$.
We say that a complex structure $J$ on $\R^{2n+2}$ \emph{preserves} a singular Riemannian foliation $\cal{F}$ of the unit sphere $\mathbb{S}^{2n+1}$ of $\R^{2n+2}$ if $\cal{F}$ admits the Hopf foliation $\{\pi_J^{-1}(p):p\in \C \P^n\}$ as a subfoliation. A completely analogous definition holds in the quaternionic setting, but in this case, each leaf of $\cal{F}$ has to be foliated by the $\mathbb{S}^3$-fibers of $\pi_\g{q}$ (instead of the $\mathbb{S}^1$-fibers of $\pi_J$).

The approach followed in~\cite{Do:tams,DG:tohoku} to address Problem~(P) relied, firstly, on the determination of the automorphism group
\[
\Aut(\cal{F})=\{A\in \mathsf{O}(r+1):A\text{ maps leaves of }\cal{F}\text{ to leaves of }\cal{F}\}
\]
of each polar foliation $\cal{F}$ of a round sphere $\mathbb{S}^r\subset\R^{r+1}$, and secondly, on the classification of some $\U{1}$ or $\SU{2}$ subgroups of $\Aut(\cal{F})$ up to certain equivalence relation. In~\cite{Do:tams}, the automorphism groups of homogeneous polar foliations of round spheres were calculated, as well as of most inhomogeneous, codimension one, polar foliations of FKM type (namely, the celebrated inhomogeneous isoparametric families of hypersurfaces  discovered by Ferus, Karcher and M\"unzner~\cite{FKM:mathz}). We recall that each one of the latter FKM foliations $\cal{F}_\cal{P}$ is determined by a space $\cal{P}=\spann\{P_0,\dots, P_m\}$ of symmetric $2l\times2l$-matrices such that $P_iP_j+P_jP_i=2\delta_{ij}\id$, for each $i=0,\dots, m$. Then, the corresponding isoparametric hypersurfaces in $\mathbb{S}^{2l-1}$ have $g=4$ constant principal curvatures with multiplicities $m_1=m_3=m$ and $m_2=m_4=l-m-1$. A word of caution is in order: with this notation (which agrees with the one in~\cite{FKM:mathz}), the multiplicities $m_1$, $m_2$ are not interchangeable, and encode different algebraic/geometric properties of the foliation (in particular, one cannot assume $m_1\leq m_2$ or $m_2\leq m_1$).

In~\cite{Do:tams}, the automorphism group $\Aut(\cal{F}_\cal{P})$ was calculated for all FKM foliations $\cal{F}_\cal{P}$ with $m_1\leq m_2$ (in the notation of~\cite{FKM:mathz} summarized above). The argument involved Clifford modules and pin/spin groups, and made use of a crucial property of most FKM foliations: if $m_1\leq m_2$, the FKM foliation determines the space $\cal{P}$ of symmetric matrices. There are eight FKM examples with $m_1>m_2$, but most of them are homogeneous or congruent to FKM examples with $m_1\leq m_2$. Thus, at the end, only two cases were left open: those with multiplicities $(m_1,m_2)=(8,7)$ in $\mathbb{S}^{31}$.

Curiously enough, when the papers~\cite{Do:tams,DG:tohoku} were written, isoparametric hypersurfaces in spheres with $g=4$ principal curvatures and multiplicities $7$, $8$ were not classified. This case was indeed the latest one to be successfully addressed in the context of the outstanding classification problem of isoparametric hypersurfaces in spheres. This was very recently achieved by Chi in~\cite{Chi:jdg}. His result states that an isoparametric family of hypersurfaces with $g=4$ principal curvatures and multiplicities $7$, $8$ in $\mathbb{S}^{31}$ must be one of the three FKM foliations with such properties: the two examples mentioned above with $(m_1,m_2)=(8,7)$, or one example with $(m_1,m_2)=(7,8)$ (which was actually discovered first by Ozeki and Takeuchi~\cite{OT}).

\medskip
In this article, we determine the automorphism group of the two exceptional FKM foliations with $(m_1,m_2)=(8,7)$ on $\mathbb{S}^{31}$ by a more direct approach than the one in~\cite{Do:tams}, namely by analyzing the lattice of closed subgroups of $\SO{32}$ that contain $\Spin{9}$ in light of some geometric features of the FKM foliations under consideration. Then we solve problem (P) above for each one these two FKM foliations $\cal{F}$.
Since the explicit description of the complex or quaternionic structures in problem (P) is a bit technical (we refer to \S\ref{subsec:8,7} for details), in the theorems below we content ourselves with indicating the maximal numbers $N_\C(\cal{F})$ and $N_\H(\cal{F})$ of complex or quaternionic structures in the conditions of problem (P).

As mentioned above,  each FKM foliation is determined by a space $\cal{P}=\spann\{P_0,\dots, P_m\}$ of symmetric matrices of the same order $2l$ such that $P_iP_j+P_jP_i=2\delta_{ij}\id$. Each such space $\cal{P}$ provides a Clifford module structure on $\R^{2l}$. We will denote by $k$ the number of its irreducible submodules, and by $\delta(m)$ the half of their common dimension (hence, $2l=2k\delta(m)$). If $m\not\equiv 0 \mod 4$, there is only one equivalence class of irreducible Clifford modules. If
$m\equiv 0 \mod 4$, there are exactly two equivalence classes of irreducible Clifford modules; in this case, let $k_+$ and $k_-$ be the number of each one of these two classes appearing in the decomposition into irreducible submodules (hence, $k=k_++k_-$). This implies that there is only one congruence class of FKM foliations with associated multiplicities $(m_1,m_2)$ when $m=m_1\not\equiv 0 \mod 4$. However, if $m=m_1\equiv 0 \mod 4$, two FKM foliations with associated multiplicities $(m_1,m_2)$ are congruent if and only if the associated sets $\{k_+,k_-\}$ coincide. Following~\cite[\S4.3]{FKM:mathz} we will denote by $\cal{F}_{(m_1,m_2)}$, $\cal{F}_{\underline{(m_1,m_2)}}$, $\cal{F}_{ \underline{\underline{(m_1,m_2)}}}\dots$ the different representatives of the congruence classes of FKM foliations with associated multiplicities $(m_1,m_2)$. Thus, for example, the representatives of the two congruence classes of FKM foliations with $(m_1,m_2)=(8,7)$ will be denoted by $\cal{F}_{(8,7)}$ (the one with $\{k_+,k_-\}=\{0,2\}$) and by $\cal{F}_{\underline{(8,7)}}$ (the one with $k_+=k_-=1$).

We can now state our main results. In the rest of this section, $\pi\colon \sph^{{a}(n+1)-1}\to\K\P^n$, with $\K\in\{\C,\H\}$ and ${a}=\dim_\R\K$, denotes a fixed Hopf fibration. We will use the letter $\cal{F}$ to refer to foliations of spheres, and $\cal{G}$ to foliations of projective spaces.

\begin{theorem}\label{th:87}
	Let $\cal{F}\in\{\cal{F}_{(8,7)},\cal{F}_{\underline{(8,7)}}\}$ be an FKM foliation with $(m_1,m_2)=(8,7)$ on~$\mathbb{S}^{31}$. Then, up to congruence, there are exactly $N_\C(\cal{F})$ (resp.\ $N_\H(\cal{F})$)  polar foliations $\cal{G}$ of $\C\P^{15}$ (resp.\ on $\H \P^7$) that pull back  under the Hopf fibration to a foliation $\pi^{-1}\cal{G}$ congruent to $\cal{F}$, where:
	\[
	\begin{array}{c@{\qquad\qquad\qquad}c}
	N_\C(\cal{F})=\begin{cases}2,& \text{if }\cal{F}=\cal{F}_{(8,7)},\\1,&  \text{if }\cal{F}=\cal{F}_{\underline{(8,7)}},
	\end{cases}
	&
		N_\H(\cal{F})=1.
	\end{array}
	\]
	All such polar foliations $\cal{G}$ of $\C\P^{15}$ and $\H\P^7$ are inhomogeneous.
\end{theorem}

As already mentioned, the results of~\cite{Do:tams} and~\cite{DG:tohoku} left open the classification problem of codimension one polar foliations of $\C\P^{15}$ and $\H\P^7$. Thanks to Theorem~\ref{th:87}, we can derive such a classification in the following corollary, where $\cal{G}^\mathrm{red}_{\K\P^k,\K\P^{n-k-1}}$ denotes the codimension one polar foliation of $\K\P^n$, $\K\in\{\C,\H\}$, whose singular leaves are totally geodesic $\K\P^k$ and $\K\P^{n-k-1}$ in $\K \P^n$, $k\in\{0,\dots, n-1\}$. We also recall that a polar foliation $\cal{G}$ of $\K\P^n$ is called irreducible if there is no totally geodesic $\K\P^k$ in $\K \P^n$, $k\in\{0,\dots, n-1\}$, that is foliated by leaves of $\cal{G}$. This is equivalent to the pullback~$\pi^{-1}\cal{G}$ being irreducible as a foliation of the corresponding sphere.

\begin{corollary}[Codimension one polar foliations of $\C\P^{15}$ and $\H\P^7$]\label{cor:15,7}
\hfill\\[-1em]

\begin{enumerate}[\rm (1)]
  \item Let $\cal{G}$ be a codimension one polar foliation  of $\C\P^{15}$. Then, $\cal{G}$ is congruent to a reducible foliation $\cal{G}^\mathrm{red}_{\C\P^k,\C\P^{\ell}}$ with $k+\ell=14$ or to one of the $N_\C(\cal{F})$ irreducible polar foliations of  $\C\P^{15}$  such that $\pi^{-1}\cal{G}$ is congruent to $\cal{F}$, for one of the FKM foliations $\cal{F}$ in Table~\ref{table:15,7}.

  \item Let $\cal{G}$ be a codimension one polar foliation  of $\H\P^7$. Then, $\cal{G}$ is congruent to a reducible foliation $\cal{G}^\mathrm{red}_{\H\P^k,\H\P^{\ell}}$ with $k+\ell=6$, or to one of the $N_\H(\cal{F})$ irreducible polar foliations of $\H\P^7$ such that $\pi^{-1}\cal{G}$ is congruent to $\cal{F}$, for one of the FKM foliations $\cal{F}$ in Table~\ref{table:15,7}.

\end{enumerate}
\end{corollary}

\begin{table}[h!]
\begin{tabular}{cccc||cc||cc}
	$\cal{F}$ & $m$ & $k$ or $k_\pm$ & $G/K$ if $\cal{F}$ homogeneous & $N_\C(\cal{F})$ & $N_\C^\mathrm{h}(\cal{F})$ & $N_\H(\cal{F})$ & $N_\H^\mathrm{h}(\cal{F})$\\
	\hline
	$\cal{F}_{(1,14)}$   & $1$ & $16$  & $\SO{18}/\SO{2}\times\SO{16}$ & $2$ & $1$ & $1$ & $0$
	\\
	$\cal{F}_{(2,13)}$ & $2$ & $8$ & $\SU{10}/\SUxU{2}{8}$ & $6$ & $1$ & $2$ & $1$
	\\
	$\cal{F}_{(3,12)}$ & $3$ & $4$ & Inhomogeneous & $2$ & $0$ & $2$ & $0$
	\\
	$\cal{F}_{(4,11)}$ & $4$ & $0,4$ & $\Sp{6}/\Sp{2}\times\Sp{4}$ &  $2$ & $0$ & $2$ & $0$
	\\
	$\cal{F}_{\underline{(4,11)}}$ & $4$ & $1,3$ & Inhomogeneous &  $2$ & $0$ & $2$ & $0$
	\\
	$\cal{F}_{\underline{\underline{(4,11)}}}$ & $4$ & $2,2$ & Inhomogeneous &  $2$ & $0$ & $2$ & $0$
	\\
	$\cal{F}_{(5,10)}$ & $5$ & $2$ & Inhomogeneous &  $2$ & $0$ & $2$ & $0$
	\\
		$\cal{F}_{(6,9)}$ & $6$ & $2$ & Inhomogeneous & $3$ & $0$ & $2$ & $0$
	\\
	$\cal{F}_{(7,8)}$ & $7$ & $2$ & Inhomogeneous & $2$ & $0$ & $1$ & $0$
	\\
	$\cal{F}_{(8,7)}$ & $8$ & $0,2$ & Inhomogeneous & $2$ & $0$ & $1$ & $0$
	\\
	$\cal{F}_{\underline{(8,7)}}$ & $8$ & $1,1$ & Inhomogeneous & $1$ & $0$ & $1$ & $0$
	\\
	$\cal{F}_{(9,6)}$ & $9$ & $1$ & $\E{6}/\Spin{10}\cdot\U{1}$ & $2$ & $1$ & $1$ & $0$
	\\ \hline
\end{tabular}
\smallskip
\caption{Irreducible codimension one polar foliations of $\C\P^{15}$ and $\H\P^7$ in terms of their pullback foliations $\cal{F}$ of $\sph^{31}$.}\label{table:15,7}
\end{table}

In Table~\ref{table:15,7}, all foliations $\cal{F}$ are of FKM type, and hence their regular leaves have $g=4$ distinct principal curvatures with multiplicity pair $(m_1,m_2)$ indicated as a subscript.
Some of the FKM foliations in Table~\ref{table:15,7} are homogeneous, and hence coincide with the orbit foliation (restricted to the unit sphere) of the isotropy representation of a symmetric space $G/K$ (of rank $2$). The number $N_\K^\mathrm{h}(\cal{F})$ of congruence classes of homogeneous polar foliations $\cal{G}$ with $\pi^{-1}\cal{G}$ congruent to $\cal{F}$ is also indicated in Table~\ref{table:15,7}.

For completeness, we include the following Corollary~\ref{cor:codim1} with the complete classification of codimension one polar foliations of complex and quaternionic projective spaces. The classification for codimension greater than one was obtained in~\cite{Do:tams} and~\cite{DG:tohoku}.

\begin{remark}
	Since any codimension one singular Riemannian foliation of a complete Riemannian manifold is polar (see~\cite[Theorem~5.24]{AB} or~\cite[\S2.4]{Lytchak:gd}), the following classes of singular Riemannian foliations of complex and quaternionic projective spaces (as well as of spheres) coincide: singular Riemannian foliations of codimension one, polar foliations of codimension one, isoparametric families of hypersurfaces. For the equivalence between polarity and isoparametricity, see Remark~\ref{rem:isop}.
\end{remark}

\begin{corollary}[Classification of codimension one polar foliations of $\C\P^n$ and $\H\P^n$]\label{cor:codim1}
	Let $\cal{G}$ be a codimension one polar foliation  on $\K\P^n$, $\K\in\{\C,\H\}$. Then, $\cal{G}$ is congruent to a reducible foliation $\cal{G}^\mathrm{red}_{\K\P^k,\K\P^{n-k-1}}$, $k\in\{0,\dots,n-1\}$, or to one of the $N_\K(\cal{F})$ irreducible polar foliations of $\K\P^n$ such that $\pi^{-1}\cal{G}$ is congruent to $\cal{F}$, for one of the foliations $\cal{F}$ listed in Tables~\ref{table:FKM} and~\ref{table:hom_c1}.
\end{corollary}

\begin{table}[h!]
	\begin{tabular}{c||cc||ccc}
		$m\mod 8$ & $N_\C$ & $k$, $k_\pm$ & $N_\H$& $k$, $k_\pm$ \\
		\hline
		\multirow{2}{*}{$0$} &  $2$ & $k_\pm$ even & $2$ & $k_\pm\equiv 0 \mod 4$
		\\
		& $1$ & otherwise & $1$ & otherwise
		\\ \hline
		\multirow{2}{*}{$1$, $7$} & $2$ & $k$ even & $2$ & $k\equiv 0\mod 4$
		\\
		& $1$ & $k$ odd & $1$ & $k\not\equiv 0\mod 4$
		\\ \hline
		\multirow{2}{*}{$2$, $6$} & \multirow{2}{*}{$2+\left[\frac{k}{2}\right]$} & \multirow{2}{*}{any $k$} & $2$ & $k$ even
		\\
		& & & $1$ & $k$ odd
		\\ \hline
		$3$, $4$, $5$ & $2$ & any $k$, $k_\pm$ & $2$ & any $k$, $k_\pm$
		\\ \hline
	\end{tabular}
	\caption{Number of congruence classes of (a fortiori, inhomogeneous) polar foliations of $\C\P^{k\delta(m)-1}$ and $\H\P^{\frac{k\delta(m)}{2}-1}$ obtained as projections of \emph{inhomogeneous} FKM foliations $\cal{F}$.}\label{table:FKM}
\end{table}

\setlength{\tabcolsep}{0.65ex}
\newcommand\Tstrut{\rule{0pt}{1.5\normalbaselineskip}}
\newcommand\Bstrut{\rule[-\normalbaselineskip]{0pt}{1.5\normalbaselineskip}}
\begin{table}[h!]
	{\footnotesize
		\bigskip
	\begin{tabular}{cccc||cccc||cccc}
		$G/K$ & $g$ & $(m_1,m_2)$ & FKM? & $n_\C$ & $N_\C$ & Cond. & $N_\C^\mathrm{h}$& $n_\H$ & $N_\H$ & Condition & $N_\H^\mathrm{h}$ \\
		\hline
		\multirow{2}{*}{$\displaystyle\frac{\SO{k+2}}{\SO{2}\times\SO{k}}$} & \multirow{3}{*}{$4$} & \multirow{3}{*}{$(1,k-2)$} & \multirow{2}{*}{Yes}  & \multirow{3}{*}{$k-1$} & \multirow{2}{*}{$1$} & \multirow{2}{*}{$k$ odd} & \multirow{3}{*}{$1$} & \multirow{2}{*}{$\frac{k}{2}-1$} & $1$ & $k\equiv 0\mod 4$ & \multirow{2}{*}{$0$}
		\\
		& & & \multirow{2}{*}{$m=1$} & & \multirow{2}{*}{$2$} & \multirow{2}{*}{$k$ even} & & & $0$ & $k\equiv 2\mod 4$
		\\
		\multirow{1}{*}{$k\geq 3$} &&&&&&&&$-$& $-$ & $k$ odd & $-$
		\\ \hline \Tstrut
		$\displaystyle\frac{\SU{k+2}}{\SUxU{2}{k}}$  &
		\multirow{2}{*}{$4$} & \multirow{2}{*}{$(2, 2k-3)$} & Yes & \multirow{2}{*}{$2k-1$} & $2+[\frac{k}{2}]$ & $k\neq 2$ & \multirow{2}{*}{$1$} & \multirow{2}{*}{$k-1$} & $2$ & $k\neq 2$ even & \multirow{2}{*}{$1$}
		\\
		$k\geq 2$ &&& \multirow{0.5}{*}{$m=2$}&& $2$ & $k=2$ & & & $1$ & $k=2$ or odd
		\\ \hline
		\multirow{2}{*}{$\displaystyle\frac{\Sp{k+2}}{\Sp{2}\times\Sp{k}}$} & \multirow{3}{*}{$4$} & \multirow{3}{*}{$(4,4k-5)$} & Yes & \multirow{3}{*}{$4k-1$} & \multirow{2}{*}{$2$} & \multirow{2}{*}{$k\geq 3$} & \multirow{3}{*}{$0$} & \multirow{3}{*}{$2k-1$} & \multirow{2}{*}{$2$} & \multirow{2}{*}{$k\geq 3$} & \multirow{3}{*}{$0$}
		\\
		\multirow{3}{*}{$k\geq 2$}&&& $m=4$, &&  &  & &&  &   &
		\\
		&&& $0\in\{k_\pm\}$&& $1$  &$k=2$&&&$1$&$k=2$&
		\\ \hline
		\multirow{3}{*}{$\displaystyle\frac{\E{6}}{\Spin{10}\cdot \U{1}}$} & \multirow{3}{*}{$4$} & \multirow{3}{*}{$(9,6)$} & Yes & \multirow{3}{*}{$15$} & \multirow{3}{*}{$2$} & & \multirow{3}{*}{$1$} & \multirow{3}{*}{$7$} & \multirow{3}{*}{$1$} & & \multirow{3}{*}{$0$}
		\\
		&&&$m=9$, &&&&&&&
		\\
		&&&$k=1$ &&&&&&&
		\\ \hline
		\Tstrut $\displaystyle\frac{\SO{10}}{\U{5}}$ & $4$ & $(4,5)$ & No & $9$ & $2$ & &  $1$ & $4$ & $0$ & & $0$
			\Bstrut
		\\ \hline
		\Tstrut $\displaystyle\frac{\G}{\SO{4}}$ & $6$ & $(1,1)$ & No & $3$ & $1$ & & $0$ & $1$ & $1$ & & $1$
		\Bstrut
		\\
		\hline
	\end{tabular}}
	\caption{Number $N_\K(\cal{F})$ of congruence classes of polar foliations of $\K\P^{n_\K}$, $\K\in\{\C,\H\}$, obtained as projections of irreducible homogeneous codimension one foliations $\cal{F}$; exactly $N_\K^\mathrm{h}(\cal{F})$ of such classes are homogeneous. Each homogeneous foliation $\cal{F}$ is represented by the symmetric space $G/K$ whose isotropy representation, restricted to the unit sphere, yields $\cal{F}$.
	A dash ($-$) indicates the nonexistence of projections due to dimension reasons.
	}
	\label{table:hom_c1}
\end{table}

Some remarks about Tables~\ref{table:FKM} and~\ref{table:hom_c1} are in order:
\begin{itemize}
\item The sets of FKM foliations and of homogeneous polar foliations of spheres intersect in three infinite families and one additional example (these correspond to the first four rows of Table~\ref{table:hom_c1}, where we inform the values of $m$ and $k$, $k_\pm$ of the associated FKM foliation).
\item The values for $N_\K$ in Table~\ref{table:FKM} are only valid for the inhomogeneous FKM foliations~$\cal{F}$. Thus, the $N_\K$ of the homogeneous FKM foliations can only be read from Table~\ref{table:hom_c1}.
\item Recall that the two singular leaves of the foliations $\cal{G}$ of $\K\P^n$ obtained as projections of $\cal{F}$ have codimensions $m_1+1$, $m_2+1$, where $(m_1,m_2)$ are the multiplicities of the $g$ distinct principal curvatures of the regular leaves of $\cal{F}$.
\item The integer $g$ coincides with the number of points of a (closed) normal geodesic to $\cal{G}$ that lie in a singular leaf of $\cal{G}$. It follows from Corollary~\ref{cor:codim1} that $g=4$ for all codimension one polar foliations of $\K \P^n$, except for the reducible foliations $\cal{G}^\mathrm{red}_{\K\P^k,\K\P^{n-k-1}}$ ($g=2$), an inhomogeneous foliation of $\C\P^3$ ($g=6$) and a homogeneous foliation of $\H\P^1\cong\sph^{4}$ ($g=6$).
\item For the FKM foliations in Table~\ref{table:FKM}, $g=4$ and the multiplicity pair is $(m_1,m_2)=\bigl(m,k\delta(m)-m-1\bigr)$, always under the assumption that the second entry is positive.
\item If $\cal{G}\subset\K\P^n$ is homogeneous, then its pullback $\pi^{-1}\cal{G}$ must be homogeneous (whence the inhomogeneity of the projected foliations in Table~\ref{table:FKM}). However, a homogeneous polar foliation $\cal{F}$ of a sphere can give rise to inhomogeneous projections.
\item The number $N_\K^\mathrm{h}(\cal{F})$ of homogeneous polar foliations of $\K\P^n$ whose pullback is congruent to $\cal{F}$ is also indicated in Table~\ref{table:hom_c1}. This number is at most one, and this happens precisely when the associated symmetric space is Hermitian (for $\K=\C$) or quaternionic-K\"ahler (for $\K=\H$).
\end{itemize}

This paper is organized as follows. In Section~\ref{sec:aut} we investigate the automorphism group of the FKM foliations with $(m_1,m_2)=(8,7)$. In Section~\ref{sec:projections} we recall the approach proposed in~\cite{Do:tams} and~\cite{DG:tohoku} to analyze the projections of polar foliations of round spheres to complex and quaternionic projective spaces via the respective Hopf fibrations (in \S\ref{subsec:approach}), and we apply this approach, in combination with the results of Section~\ref{sec:aut}, to prove our main results (in~\S\ref{subsec:8,7}).

\section{The automorphism group of the FKM foliations with $(m_1,m_2)=(8,7)$}\label{sec:aut}
The aim of this section is to compute the connected component of the identity of the automorphism group of both FKM foliations with multiplicities $(m_1,m_2)=(8,7)$.

Recall from the introduction that an FKM foliation is determined by a vector space $\cal{P}=\spann\{P_0,\dots, P_m\}$ of symmetric matrices of order $2l$ satisfying the Clifford relations $P_iP_j+P_jP_i=2\delta_{ij}\id$.
Attached to $\cal{P}$ we have a Clifford algebra representation $\chi\colon\mathsf{Cl}^*_{m+1}\to \mathrm{End}(\R^{2l})$ so that $\cal{P}=\chi(\R^{m+1})$, where  $\mathsf{Cl}^*_{m+1}$ is the Clifford algebra associated with $\R^{m+1}$ (following the notation of~\cite[\S{}I.3]{LM}).

Let $\R^{2l}=\bigoplus_{r=1}^k \g{d}_r$ be the decomposition of $\R^{2l}$ into irreducible Clifford modules for the representation $\chi$. If $m\not\equiv 0\mod 4$, all irreducible submodules $\g{d}_r$ are equivalent, whereas if $m\equiv 0\mod 4$ each $\g{d}_r$ is equivalent to one of two possible irreducible modules, say $\g{d}_+$ and $\g{d}_-$; in the latter case, we let $k_\pm$ be the number of copies of $\g{d}_\pm$ in the above decomposition into irreducible submodules, respectively, and hence $k=k_++k_-$. In this paper, as in~\cite{FKM:mathz}, we let $2\delta(m)$ be the dimension of any irreducible $\mathsf{Cl}^*_{m+1}$-module. Therefore, we have $l=k\delta(m)$.

Provided that $k\delta(m)-m-1\geq 1$, the FKM foliation $\cal{F}$ of the unit sphere $\sph^{2l-1}$ of $\R^{2l}$ associated with $\cal{P}$ is given by the level sets of $F\vert_{\sph^{2l-1}}$, where $F\colon \R^{2l}\to\R$ is the Cartan-M\"unzner polynomial
\[
F(x)=\langle x, x\rangle^2-2\sum_{i=0}^m\langle P_i x, x\rangle^2.
\]
In this case, $\cal{F}$ is an isoparametric family of hypersurfaces with $g=4$ constant principal curvatures and multiplicity pair $(m_1,m_2)=(m,k\delta(m)-m-1)$. This isoparametric foliation of $\sph^{2l-1}$ has two focal leaves, $M_1$ and $M_2$, of respective codimensions $m_1+1$ and $m_2+1$ in $\sph^{2l-1}$.

Up to congruence, an FKM foliation is determined by the positive integers $m$ and $k$ if $m\not\equiv 0\mod 4$, or by the positive integer $m$ and the unordered pair $\{k_+,k_-\}$ if $m\equiv 0\mod 4$. In the first case, we denote by $\cal{F}_{(m_1,m_2)}$ any representative of the unique congruence class with $m=m_1$ and $k=(m_2+m+1)/\delta(m)$. In the second case, we denote by $\cal{F}_{(m_1,m_2)}$, $\cal{F}_{\underline{(m_1,m_2)}}$, $\cal{F}_{ \underline{\underline{(m_1,m_2)}}}\dots$ the representatives of the $[k/2]+1$ congruence classes with $m=m_1$ and $k=(m_2+m+1)/\delta(m)$, where the number of underlinings is $\min\{k_\pm\}$. Two FKM foliations are congruent if and only if they share the same multiplicity pair and the same value $\min\{k_\pm\}$, with the only following exceptions: $\cal{F}_{(2,1)}=\cal{F}_{(1,2)}$,
$\cal{F}_{\underline{(4,3)}}=\cal{F}_{(3,4)}$, $\cal{F}_{(5,2)}=\cal{F}_{(2,5)}$, $\cal{F}_{(6,1)}=\cal{F}_{(1,6)}$.
We refer to~\cite[\S4.3]{FKM:mathz} or~\cite[\S5.1]{Chi} for the table of possible  multiplicities with $m$ and $k$ small and additional information.
\medskip

We need to recall some known facts about certain subgroups of the automorphism group of FKM foliations. The connected Lie subgroup of $\SO{2l}$ with Lie algebra $\spann\{P_iP_j:i\neq j\}\cong\g{so}(m+1)$ turns out to be isomorphic to $\Spin{m+1}$. This Lie group $\Spin{m+1}$ is a subgroup of the automorphism group $\Aut(\cal{F})$ of the FKM foliation $\cal{F}$, cf.~\cite[\S6.2]{FKM:mathz}. The group
\[
\mathsf{U}^+(\cal{P})=\{A\in \OG{2l}:AP=PA\text{ for all } P\in\cal{P}\}
\]
is also a closed subgroup of $\Aut(\cal{F})$. Hence, the map $\Psi\colon \Spin{m+1}\times \mathsf{U}^+(\cal{P})\to\Aut(\cal{F})$, $\Psi(Q,A)= QA$, is a Lie group homomorphism with finite kernel (indeed, $\ker\Psi$ is a subgroup of the center of $\Spin{m+1}$). The group $\mathsf{U}^+(\cal{P})$ was computed in~\cite[Theorem~3.4]{Do:tams}, see also~\cite{Riehm}, for all FKM foliations.
In particular, for $m\equiv 0\mod 8$ we have $\mathsf{U}^+(\cal{P})\cong \OG{k_+}\times\OG{k_-}$.
In this case, the action of the subgroup $\Spin{m+1}\cdot\mathsf{U}^+(\cal{P})$ of $\Aut(\cal{F})$ on $\R^{2l}$ is given by the direct sum of two tensor product representations, namely $(\g{d}_+\otimes\R^{k_+})\oplus (\g{d}_-\otimes\R^{k_-})$, where $\Spin{m+1}$ acts on the left factors $\g{d}_\pm$ via the real spin representation and $\OG{k_\pm}$ acts on the corresponding right factor $\R^{k_\pm}$ in the standard way.

From now on, we will focus on the case we are interested in, namely on the FKM foliations $\cal{F}_{(8,7)}$ and $\cal{F}_{\underline{(8,7)}}$ of $\sph^{31}\subset\R^{32}$. According to the previous discussion, we know that $\Spin{9}\cdot \OG{2}$ is a subgroup of $\Aut(\cal{F}_{(8,7)})$, and $\Spin{9}\cdot(\OG{1}\times\OG{1})$ is a subgroup of $\Aut(\cal{F}_{\underline{(8,7)}})$. Taking connected components, we get the subgroups
\begin{equation}\label{eq:inclusions}
\Spin{9}\cdot \SO{2}\subset\Aut(\cal{F}_{(8,7)})^0\qquad \text{and}\qquad \Spin9\subset\Aut(\cal{F}_{\underline{(8,7)}})^0.
\end{equation}
The aim of the rest of this section is to show that these inclusions are actually equalities.

\begin{remark}\label{rem:O(2)}
	For our purposes, we will not need to determine the full automorphism groups of  $\cal{F}_{(8,7)}$ and $\cal{F}_{\underline{(8,7)}}$. However, apart from their connected components (Proposition~\ref{prop:2cases} below), we will need to use the above mentioned fact that  $\Spin{9}\cdot \OG{2}\subset \Aut(\cal{F}_{(8,7)})$.
\end{remark}

In the following lemma, we discuss the connected subgroups of $\SO{32}$ that contain a $\Spin{9}$ acting diagonally (componentwise) on $\R^{32}=\R^{16}\oplus\R^{16}$. Note that both $\Spin{9}$ appearing in~\eqref{eq:inclusions} are diagonal in this sense. The specific diagonal embedding will not play a role in our arguments, since all of them are conjugate in $\mathsf{O}(32)$, and hence their actions on $\R^{32}$ have the same cohomogeneities.

\begin{lemma}\label{lemma:lattice}
Let $H$ be a closed connected subgroup of~$\SO{32}$ containing a diagonal~$\Spin{9}$. Then $H$ is conjugate in $\mathsf{O}(32)$ to one of the subgroups given in Figure~\ref{T:SubgrSO32}, where the superscripts indicate the cohomogeneity of the $H$-action on~$\R^{32}$.
In particular, if $H$ acts on~$\R^{32}$ with cohomogeneity at least~$3$, we have:
\begin{enumerate}[\rm(i)]
\item $H$ does not properly contain~$\Spin{9}\cdot\SO{2}$.
\item If $H$ properly contains the diagonal~$\Spin{9}$, then $H$ is conjugate to a diagonal $\SO{16}$ or to \mbox{$\Spin{9}\cdot\SO{2}$}.
\end{enumerate}

\begin{figure}[h]
\begin{picture}(200,180)
\put(-75,90)
{\begin{tabular}{c@{\hspace{4ex}}l@{\hspace{4ex}}l@{\hspace{4ex}}c}
&   \multicolumn{2}{l}{$\SO{32}^1$} & \\
& & & \\
 $\SO{16}{\times}\SO{16}^2$ &  & $\U{16}^1$ & \\
& & & \\
$\Spin{9}{\times}\SO{16}^2$ & $\SU{16}^1$ & $\SO{16}\cdot\SO{2}^2$  & $\Spin{10}\cdot \SO{2}^2$ \\
& & & \\
$\Spin{9}{\times}\Spin{9}^2$ & $\SO{16}^3$ & $\Spin{9}\cdot\SO{2}^3$ & $\Spin{10}^2$\\
& & & \\
&\multicolumn{2}{l}{\hspace{6ex}$\Spin{9}^4$} & \\
& & & \\
\end{tabular}}
\put(60,164){\line(2,-1){36}} \put(55,164){\line(-5,-1){75}}
\put(124,130){\line(6,-1){110}}
\put(118,130){\line(0,-1){20}}\put(114,130){\line(-3,-1){60}}\put(-28,130){\line(0,-1){20}}
\put(4,130){\line(2,-3){37}}
\put(-28,94){\line(0,-1){20}}
\put(118,94){\line(0,-1){20}}
\put(48,94){\line(0,-1){20}}
\put(240,94){\line(0,-1){20}}
\put(114,94){\line(-3,-1){60}}
\put(230,94){\line(-5,-1){90}}
\put(-8,58){\line(3,-1){70}}
\put(114,58){\line(-1,-1){20}}
\put(58,58){\line(1,-1){20}}
\put(230,58){\line(-4,-1){115}}
\end{picture}
\caption{Lattice of connected subgroups
of~$\SO{32}$ containing a diagonal~$\Spin{9}$, up to conjugation in $\mathsf{O}(32)$. Superscripts denote the cohomogeneity of the action on $\R^{32}$.} \label{T:SubgrSO32}
\end{figure}
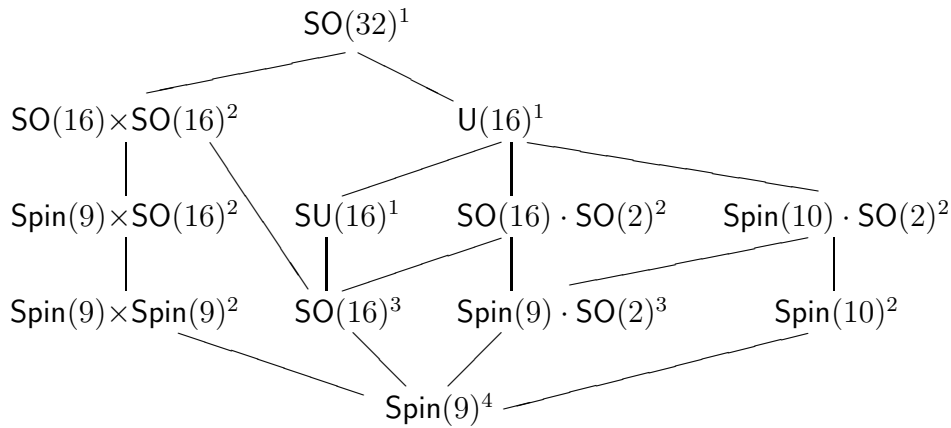
\end{lemma}

\begin{proof}
We use~\cite[Chapter~1]{Dynkin2}, see also~\cite[Theorems~2.1 and~2.2]{kollross}, to determine maximal (proper) connected subgroups of compact Lie groups.

It follows from~\cite[Theorem~2.2]{kollross} that the maximal connected subgroups of~$\SO{32}$ containing~ a diagonal $\Spin9$ are $\SO{16} \times \SO{16}$ and $\U{16}$. (Note that the only irreducible representation of real type of a simple compact Lie group of degree~$32$ is the standard representation of~$\SO{32}$ itself.) We distinguish the following cases.
	\smallskip
	
\emph{1. H is a proper subgroup of $\SO{16} \times \SO{16}$}.
By~\cite[Theorems~2.1 and~2.2]{kollross}, after conjugation in $\mathsf{O}(32)$, $H$ is contained in $\Spin9 \times \SO{16}$ or in a diagonal subgroup $\{(g,\phi(g)) \mid g \in \SO{16} \}$, where $\phi$ is an automorphism of~$\SO{16}$; but any automorphism of $\SO{16}$ is given by conjugation in $\mathsf{O}(16)$.  We have to consider the connected subgroups in both cases.
	\begin{enumerate}
\item[\emph{1a}.] The only maximal connected subgroup of $\Spin9 \times \SO{16}$ containing a diagonal $\Spin9$  is $\Spin9\times\Spin9$. This product has the diagonal $\Spin9$ as a maximal connected subgroup.
\item[{\emph{1b}.}] Since $\Spin9$ is maximal in $\SO{16}$, we readily have that the only proper connected subgroup of a diagonal $\SO{16}$ in $\SO{32}$ containing a diagonal $\Spin{9}$ is this diagonal $\Spin{9}$ itself.
\end{enumerate}

	\emph{2. $H$ is a proper subgroup of $\U{16}$}.
	The maximal connected subgroups of~$\U{16}$ are $\SU{16}$ and groups of the form~$G \cdot \SO2$, where $G$ is a maximal connected subgroup of~$\SU{16}$. The only maximal connected subgroups of~$\U{16}$ that contain a diagonal~$\Spin9$ are $\SO{16} \cdot \SO2$, $\Spin{10} \cdot \SO2$ and $\SU{16}$.
	\begin{enumerate}
	\item [{\emph{2a}.}]
	The maximal connected subgroups of $\SO{16} \cdot \SO2$ that contain a diagonal~$\Spin9$ are the following: $\Spin9 \cdot \SO2$ and a diagonal $\SO{16}$. In both cases, $\Spin9$ is the only maximal connected subgroup containing $\Spin9$.
	\item[{\emph{2b}.}]
	The maximal connected subgroups of $\Spin{10} \cdot \SO2$ that contain a diagonal~$\Spin9$ are the following: $\Spin{10}$ and $\Spin9 \cdot \SO2$. The only maximal connected subgroup of~$\Spin{10}$ which contains a diagonal~$\Spin9$ is~$\Spin9$ itself.
	\item[{\emph{2c}.}]
	Finally, we have to consider the case where $H$ is contained in~$\SU{16}$. Its only maximal connected
	subgroup that contains $\Spin9$ is $\SO{16}$.
	\end{enumerate}
	
	The cohomogeneities of the actions of the various groups $H$ on $\R^{32}$ are well known~in most cases.  It follows from~\cite[p.~21, Figure~(ii)]{HL:jdg} that $\Spin{9}\cdot\SO{2}$ acts with cohomogeneity~$3$, and it can be easily determined that $\Spin{9}$ acts with cohomogeneity~$4$. The actions of $\Spin{10}\cdot \SO{2}$ and of $\Spin{10}$ are orbit equivalent to the $s$-representation of the rank $2$ symmetric space $\E{6}/\Spin{10}\cdot\SO{2}$ (see~\cite{EH:mathz}), and hence of cohomogeneity~$2$.
\end{proof}

\begin{proposition}\label{prop:2cases}
	We have $\Aut(\cal{F}_{(8,7)})^0=\Spin{9}\cdot \SO{2}$ and $\Aut(\cal{F}_{\underline{(8,7)}})^0=\Spin9$.
\end{proposition}
\begin{proof}
	Since both foliations $\cal{F}_{(8,7)}$ and $\cal{F}_{\underline{(8,7)}}$ are inhomogeneous, the connected component of the identity of their automorphism group acts with cohomogeneity at least $3$ on $\R^{32}$. Thus, the equality for $\Aut(\cal{F}_{(8,7)})^0$ readily follows from~\eqref{eq:inclusions} and Lemma~\ref{lemma:lattice}~(i). Furthermore, Lemma~\ref{lemma:lattice}~(ii) implies that $\Aut(\cal{F}_{\underline{(8,7)}})^0$ is either $\SO{16}$, $\Spin{9}\cdot \SO{2}$ or $\Spin{9}$. The proposition will be proved once we discard the first two possibilities.
	
	\smallskip
	\emph{Case $\SO{16}$.} Up to equivalence, an arbitrary diagonal action of $\SO{16}$ on $\R^{32}=\R^{16}\oplus\R^{16}$ is of the form
	$A\cdot (v,w):=(Av,Aw)$,
	for $A\in\SO{16}$, $v,w\in\R^{16}$.
	The principal orbits of such an action on $\sph^{31}$ have dimension $29$, and its singular orbits are of dimension $15$. The points of $\mathbb{S}^{31}\subset\R^{32}$ lying at singular orbits are precisely those of the form $(\lambda v,\mu v)$, where $v\in\mathbb{S}^{15}$ and  $\lambda^2+\mu^2=1$. This constitutes a $16$-dimensional submanifold of $\mathbb{S}^{31}$. If $\Aut(\cal{F}_{\underline{(8,7)}})^0$ were the subgroup of $\SO{32}$ given by the image of the representation above, the $22$-dimensional focal leaf $M_1$ of $\cal{F}_{\underline{(8,7)}}$ would have to be foliated by the $15$-dimensional (singular) orbits of the action of $\Aut(\cal{F}_{\underline{(8,7)}})^0\cong\SO{16}$. However, as shown above, the union of these orbits has dimension $16$, which yields a contradiction.
	
	\smallskip
	\emph{Case $\Spin{9}\cdot\SO{2}$.} The action of this group on $\mathbb{S}^{31}$ is also of cohomogeneity two. Its orbit structure was described in~\cite[p.~21, Figure~(ii)]{HL:jdg}. Apart from its $29$-dimensional principal orbits, it has infinitely many singular orbits diffeomorphic to $(\Spin{9}\cdot \SO{2})/\G$ (dimension $23$), to $\Spin{9}/\SU{3}$ (dimension $28$) and to $(\Spin{9}\cdot \SO{2})/\SU{4}$ (dimension $22$), and exactly one singular orbit diffeomorphic to each one of the spaces $\Spin{9}/\G$ (dimension $22$),  $(\Spin{9}\cdot \SO{2})/\Spin{7}$ (dimension $16$) and $\Spin{9}/\SU{4}$ (dimension $21$). The $22$-dimensional focal leaf $M_1$ of $\cal{F}_{\underline{(8,7)}}$ is known to be inhomogeneous, as shown in~\cite[Satz~6.6~(i)]{FKM:mathz}. Therefore, if $\Aut(\cal{F}_{\underline{(8,7)}})^0\cong\Spin{9}\cdot\SO{2}$, then $M_1$ should be the union of some of the orbits described before, with dimensions at most $21$. But this is impossible, since there are only two such orbits.
\end{proof}

\section{Analysis of the projected foliations}\label{sec:projections}
In this section we determine the congruence classes of polar foliations of $\C \P^{15}$ and $\H \P^7$ that pull back (under the corresponding Hopf fibrations) to an FKM foliation congruent to $\cal{F}_{(8,7)}$ or $\cal{F}_{\underline{(8,7)}}$. This will allow us to prove Theorem~\ref{th:87} and Corollaries~\ref{cor:15,7} and~\ref{cor:codim1}.

\subsection{General approach to the classification of polar foliations of $\C\P^n$ and $\H \P^n$}\label{subsec:approach}

We will start by establishing the setup and notation of the approach proposed in~\cite{Do:tams} and \cite{DG:tohoku}. In order to allow for a unified discussion as much as possible, we will denote $\K\in\{\C,\H\}$, ${a}=\dim_\R \K$, and consider the Euclidean space $V=\R^{{a}(n+1)}$, its unit sphere $\sph(V)=\sph^{{a} (n+1)-1}$, and the Hopf fibration $\pi\colon \sph(V)\to\K \P^n$. The starting point of the approach is the following well-known fact, see~\cite[Proposition~2.1]{Do:tams}, \cite[Proposition~2.1]{DG:tohoku} and \cite[Proposition~9.1]{Lytchak:gafa}.

\begin{fact}\label{fact:starting}
	Let $\cal{G}$ be a singular Riemannian foliation of $\K \P^n$. Then $\cal{G}$ is polar if and only if $\pi^{-1}\cal{G}$ is a polar foliation of $\sph(V)$.
\end{fact}

\begin{remark}[Polar versus isoparametric foliations]\label{rem:isop}
	An isoparametric foliation is a polar foliation whose regular leaves have parallel mean curvature, or equivalently, a polar foliation whose regular leaves 	are isoparametric submanifolds in the sense of Heintze-Liu-Olmos~\cite{HLO}.
	As noticed in~\cite[Remark~2.2]{DG:tohoku}, polar and isoparametric foliations constitute the same subclass of singular Riemannian foliations not only on round spheres, but also on complex and quaternionic projective spaces $\K\P^n$. This has been greatly generalized by Liu and Radeschi~\cite{LiuRadeschi}, who proved that a singular Riemannian foliation of a simply connected symmetric space of non-negative curvature is polar if and only if it is isoparametric. However, whereas for round spheres and complex projective spaces we know that an isoparametric submanifold is an open subset of a leaf of an isoparametric (or polar) foliation filling the whole space, this seems to be an open problem in $\H \P^n$ (except for isoparametric hypersurfaces), see~\cite[Remark~2.3]{DG:tohoku}, and even more for other symmetric spaces of compact type.
\end{remark}

In view of Fact~\ref{fact:starting}, the classification of polar foliations of $\K\P^n$ can theoretically be obtained by determining which polar foliations of the round sphere $\sph(V)$ are subfoliated by the fibers of the Hopf fibration, and then analyzing the congruence of the projected foliations to $\K\P^n$. However, two congruent foliations $\cal{F}$, $\cal{F}'=A\cal{F}$ on $\sph(V)$, with $A\in\OG{V}$, may yield non-congruent foliations $\pi(\cal{F})$, $\pi(\cal{F}')$ on $\K\P^n$. Thus, one should fix a representative $\cal{F}$ of a congruence class of polar foliations of $\sph(V)$, determine the set $\cal{A}_\cal{F}$ of those $A\in\OG{V}$ such that $A\cal{F}$ is subfoliated by the Hopf fibration, and then analyze when $\pi(A\cal{F})$ and $\pi(B\cal{F})$ are congruent foliations of $\K\P^n$, for any $A,B\in\cal{A}_\cal{F}$.

Instead of this, we will follow the (equivalent but more convenient) approach proposed in~\cite{Do:tams} and~\cite{DG:tohoku} (which goes back to Xiao~\cite{Xiao}) consisting in studying all Hopf fibrations that subfoliate a given representative $\cal{F}$ of a congruence class of polar foliations of $\sph(V)$. This leads to the study of complex or quaternionic structures that preserve a given foliation $\cal{F}$ of $\sph(V)$.

As mentioned in the introduction just below Problem (P), a Hopf fibration $\sph(V)\to\K \P^n$ is determined by a complex structure $J$ or a quaternionic structure $\g{q}$ on $V$, depending on whether $\K=\C$ or $\K=\H$. Thus, we write $\pi_J$ or $\pi_\g{q}$ for the corresponding Hopf fibration. Given a singular Riemannian foliation $\cal{F}$ of the unit sphere $\sph(V)$ of $V$, we say that a complex structure $J$ or a quaternionic structure $\g{q}$ on $V$ preserves $\cal{F}$ if $\cal{F}$ is the pullback of a singular Riemannian foliation $\cal{G}$ of $\K\P^n$ under the Hopf fibration $\pi_J$ or $\pi_\g{q}$, respectively. Equivalently, $J$ or $\g{q}$ preserves $\cal{F}$ if $\cal{F}$ is subfoliated by the leaves of the Hopf fibration $\pi_J$ or $\pi_\g{q}$, that is, if each leaf of $\cal{F}$ is a union of Hopf fibers of $\pi_J$ or $\pi_\g{q}$, respectively.

Let $\cal{F}$ be a singular Riemannian foliation of $\sph(V)$ whose leaves are closed. Consider an effective representation $\rho\colon K\to\SO{V}$ of a Lie group $K$ such that $\rho(K)$ is the maximal connected subgroup of $\SO{V}$ leaving each leaf of $\cal{F}$ invariant. (Of course, one can take directly $K$ as such a subgroup of $\SO{V}$ and $\rho$ as the inclusion map.) Since $\cal{F}$ is closed, $K$ is compact. Let $\rho_*\colon\g{k}\to\g{so}(V)$ be the Lie algebra homomorphism determined by $\rho$.

\begin{fact}\label{fact:preserve}
 A complex structure $J$ on $V$ preserves $\cal{F}$ if and only if $J=\rho_*(X)$ for some $X\in\g{k}$.
 A quaternionic structure $\g{q}$ on $V$ preserves $\cal{F}$ if and only if $\g{q}=\rho_*(\g{s})$ for some Lie subalgebra $\g{s}$ of $\g{k}$ (isomorphic to $\g{su}(2)$).
\end{fact}

Let $\cal{J}$ denote the subset of the elements $X\in\g{k}$ such that $\rho_*(X)$ is a complex structure on $V$ preserving $\cal{F}$. Similarly, we denote by $\cal{S}$ the set of all Lie subalgebras $\g{s}$ of $\g{k}$ such that $\rho_*(\g{s})$ is a quaternionic structure on $V$ preserving $\cal{F}$. Note that $\cal{J}$ and $\cal{S}$ depend on $\cal{F}$, but we prefer not to overload the notation. Clearly, $\{\rho_*(X):X\in\cal{J}\}$ (resp.\ $\{\rho_*(\g{s}):\g{s}\in\cal{S}\}$) is the set of all complex (resp.\ quaternionic) structures on $V$ that preserve $\cal{F}$.

\begin{fact}\label{fact:AdK_invariant}
	$\cal{J}$ and $\cal{S}$ are $\Ad(K)$-invariant.
\end{fact}

This will allow us to simplify the study of $\cal{J}$ and $\cal{S}$ by considering a (fixed but arbitrary) maximal abelian subalgebra $\g{t}$ of $\g{k}$. Thus, $\cal{J}$ will be known once $\cal{J}\cap \g{t}$ is determined; and $\cal{S}$ will be known once the set $\cal{S}_\g{t}=\{\g{s}\in\cal{S}:\g{s}\cap\g{t}\neq 0\}$  is determined.

Consider the complexification $\rho_*^\C\colon \g{k}^\C\to\g{gl}(V^\C)$ of $\rho_*$. Recall that a weight of the complex representation $\rho_*^\C$ is a covector $\lambda\in(\g{t}^\C)^*$ such that the subspace $V_\lambda=\{v\in V^\C:\rho_*^\C(T)v=i\lambda(T)v,\text{ for all }T\in\g{t}\}$ is non-zero. 

\begin{fact}\label{fact:pm1}
	Let $T\in\g{t}$. Then $\rho_*(T)$ is a complex structure on $V$ (preserving $\cal{F}$, due to Fact~\ref{fact:preserve}) if and only if $\lambda(T)\in\{\pm 1\}$ for every weight $\lambda$ of $\rho_*^\C$.
\end{fact}
The previous fact provides a computational method to determine $\cal{J}\cap\g{t}$ once we know the weights of $\rho^\C_*$. Now, the problem in the quaternionic setting can be reduced to the problem in the complex case, that is, $\cal{S}_\g{t}$ is given by the $\g{su}_2$-subalgebras of $\g{k}$ that intersect $\cal{J}\cap\g{t}$. This is due to the following:
\begin{fact}\label{fact:su_2_vs_complex}
	Let $\g{s}$ be an $\g{su}_2$-subalgebra of $\g{k}$. The following conditions are equivalent:
	\begin{enumerate}[\rm(i)]
		\item $\g{s}\in\cal{S}$,
		\item there exists $X\in\g{s}\cap\cal{J}$,
		\item for any one-dimensional subspace $\ell$ of $\g{s}$, we have $\ell\cap\cal{J}=\{\pm X_\ell\}\neq \emptyset$.
\end{enumerate}
\end{fact}

Once the complex/quaternionic structures that preserve a given foliation $\cal{F}$ of $\sph(V)$ are known, the problem that arises is to study the congruence of the corresponding projected foliations of $\K\P^n$.

Given two complex structures $J_1$ and $J_2$ (or quaternionic structures $\g{q}_1$ and $\g{q}_2$) on $V$, they induce (formally) two complex (or quaternionic) projective spaces $\K \P^n_1$ and $\K \P^n_2$, and two associated Hopf fibrations $\pi_i\colon \sph(V)\to\K\P_i^n$, $i=1,2$. Each one of these is given by the Riemannian submersion corresponding to the quotient of $\sph(V)$ by the action of the $\U{1}$-subgroup of $\SO{V}$ with Lie algebra $\R J_i$, $i=1,2$ (or of the $\SU{2}$-subgroup of $\SO{V}$ with Lie algebra $\g{q}_i$, $i=1,2$). Of course, $\K\P^n_1$ and $\K\P^n_2$ are isometric, and it makes sense to talk about congruence of a submanifold in $\K\P^n_1$ and a submanifold in $\K\P^n_2$. Since all isometries between $\K\P^n_1$ and $\K\P^n_2$ descend from orthogonal transformations $A\in\OG{V}$ such that $AJ_1A^{-1}=\pm J_2$ if $\K=\C$, or such that $A\g{q}_1 A^{-1}= \g{q}_2$ if $\K=\H$, we deduce the following:
\begin{fact}
	Let $J_1$, $J_2$ be complex structures (resp.\ $\g{q}_1$, $\g{q}_2$ be quaternionic structures) preserving $\cal{F}$. Then $\pi_1(\cal{F})$ and $\pi_2(\cal{F})$ are congruent foliations if and only if there exists $A\in\Aut(\cal{F})=\{A\in\OG{V}:A\text{ maps leaves of }\cal{F}\text{ to leaves of }\cal{F}\}$ such that $AJ_1A^{-1}\in\{\pm J_2\}$ (resp.\ $A\g{q}_1 A^{-1}= \g{q}_2$).
\end{fact}

In view of this, we introduce an equivalence relation $\sim$ in the set $\cal{J}$ (or $\cal{S}$) that parameterizes via $\rho_*$ the complex (or quaternionic) structures on $V$ that preserve $\cal{F}$. Thus, given $X_1$, $X_2\in\cal{J}$ we declare $X_1\sim X_2$ if the associated projections $\pi_1(\cal{F})$ and $\pi_2(\cal{F})$ are congruent, or equivalently, if there exists $A\in\Aut(\cal{F})$ such that $A\rho_*(X_1)A^{-1}\in\{\pm\rho_*(X_2)\}$. Similarly, given $\g{s}_1$, $\g{s}_2\in\cal{S}$ we declare $\g{s}_1\sim \g{s}_2$ if $\pi_1(\cal{F})$ and $\pi_2(\cal{F})$ are congruent, or equivalently, if there exists $A\in\Aut(\cal{F})$ such that $A\rho_*(\g{s}_1)A^{-1}=\rho_*(\g{s}_2)$. Thus, the maximal numbers of complex/quaternionic structures that preserve $\cal{F}$ and yield mutually non-congruent projections on the corresponding complex/quaternionic projective spaces (see Problem (P) in the introduction) are nothing but the cardinals of the quotient sets $\cal{J}/\!\sim$ and $\cal{S}/\!\sim$:
\begin{equation}\label{eq:N_CH}
	N_\C(\cal{F})=|\cal{J}/\!\sim| \qquad\text{and}\qquad N_\H(\cal{F})=|\cal{S}/\!\sim|.
\end{equation}
These numbers coincide with the maximal numbers of mutually non-congruent polar foliations of a complex or quaternionic projective space that pull back to a foliation congruent to $\cal{F}$ under a fixed Hopf fibration.

Any automorphism $A\in\Aut(\cal{F})\subset\OG{V}$ can be extended to a transformation of $\g{k}\oplus V$ in the following way. First, we associate to $A$ the Lie algebra automorphism $\varphi_A\in\Aut(\g{k})$ such that $A\rho_*(X)A^{-1}=\rho_*(\varphi_A(X))$ for all $X\in\g{k}$. Second, we define $\Phi_A\in \GL{\g{k}\oplus V}$ by $\Phi_A\vert_{\g{k}}=\varphi_A$ and $\Phi_A\vert_V=A$. We denote by $\Aut(\g{k},\cal{F})$ the compact group of all transformations $\Phi_A$, for any $A\in\Aut(\cal{F})$. Then, we have:
\begin{fact}\label{fact:aut_k,F}
	$\cal{J}$ and $\cal{S}$ are invariant under $\Aut(\g{k},\cal{F})$. Moreover, if $X_1$, $X_2\in\cal{J}$ (resp.\ $\g{s}_1$, $\g{s}_2\in\cal{S}$), then $X_1\sim X_2$ (resp.\ $\g{s}_1\sim\g{s}_2$) if and only if there exists $\Phi\in\Aut(\g{k},\cal{F})$ such that $\Phi(X_1)\in\{\pm X_2\}$ (resp.\ $\Phi(\g{s}_1)=\g{s}_2$).
\end{fact}
The study of the congruence problem amounts to the analysis of the action of $\Aut(\g{k},\cal{F})$ on $\cal{J}$ or $\cal{S}$. As for the determination of $\cal{J}$ or $\cal{S}$, this can be simplified by restricting to a maximal abelian subalgebra $\g{t}$ of $\g{k}$, or to a closed Weyl chamber $\bar{C}$ in $\g{t}$ determined by a choice of simple roots for the pair $(\g{k},\g{t})$; see below Fact~\ref{fact:aut_Jt} for the precise definition of $\bar{C}$. We denote $\cal{S}_\g{t}=\{\g{s}\in\cal{S}:\g{s}\cap\g{t}\neq 0\}$ and $\cal{S}_{\bar{C}}=\{\g{s}\in\cal{S}:\g{s}\cap\bar{C}\neq 0\}$. Then, it follows essentially from Fact~\ref{fact:AdK_invariant} that:
\begin{fact}\label{fact:equivalences}
	We have $\cal{J}/\!\sim\,\cong\cal{J}\cap\g{t}/\!\sim\, \cal{J}\cap\bar{C}/\!\sim$ and $\cal{S}/\!\sim\,\cong\cal{S}_\g{t}/\!\sim\,\cong\cal{S}_{\bar{C}}/\!\sim$.
\end{fact}

Here, we also denote by $\sim$ the restriction of the equivalence relation $\sim$ (initially defined in $\cal{J}$ and $\cal{S}$) to the subsets $\cal{J}\cap\g{t}$, $\cal{J}\cap\bar{C}$ of $\cal{J}$, and $\cal{S}_\g{t}$ and $\cal{S}_{\bar{C}}$ of $\cal{S}$.
Moreover, determining $\cal{S}_{\bar{C}}/\!\sim$ can be reduced to determining $\cal{J}\cap\bar{C}/\!\sim$ and deciding which representatives $T$ of the classes $[T]$ in $\cal{J}\cap\bar{C}/\!\sim$ belong to $\g{su}_2$-subalgebras of $\g{k}$. This follows from Fact~\ref{fact:su_2_vs_complex} and the following:
\begin{fact}\label{fact:injective}
	There is an injective map $\iota\colon \cal{S}_{\bar{C}}/\!\sim\to \cal{J}\cap\bar{C}/\!\sim$ that sends a class $[\g{s}]$ with $\g{s}\in\cal{S}_{\bar{C}}$ to the class $[T]$, where $T$ is the unique (possibly up to sign) element in $\g{s}\cap\cal{J}\cap\bar{C}$.
\end{fact}

The fact that $\iota$ is well-defined essentially follows from Fact~\ref{fact:su_2_vs_complex} (iii), whereas the injectivity is a consequence of Dynkin's uniqueness result (up to conjugation) of $\g{su}_2$-subalgebras of a compact Lie algebra $\g{k}$ containing a given one-dimensional subspace of $\g{t}$ (see~\cite[Theorem~8.1]{dynkin1} or~\cite[Theorem~7]{Vogan}).

In order to complete the description of our theoretical approach, we have to explain how to deal with the complex case, namely, how to determine $\cal{J}/\!\sim\cong\cal{J}\cap\bar{C}/\!\sim$. For this, one needs to introduce some additional ingredients.

Let us denote:
\begin{align*}
	\Aut_\cal{F}(\Delta_\g{k},\Delta_V)&=\{\varphi\vert_\g{t}:\varphi\in\Aut(\g{k},\cal{F}),\,\varphi(\g{t})=\g{t}\},
	\\
	\Aut_\cal{F}^\pm(\Delta_\g{k},\Delta_V)&=\text{subgroup of } \GL{\g{t}} \text{ generated by }\Aut_\cal{F}(\Delta_\g{k},\Delta_V)\text{ and } -\id_\g{t}.
\end{align*}
This notation is justified by the fact the restrictions to $\g{t}$ of elements of $\Aut(\g{k},\cal{F})$ belong to the group $\Aut(\Delta_\g{k},\Delta_V)$ of transformations of $\g{t}$ that leave the set of coroots $\{H_\alpha:\alpha\in\Delta_\g{k}\}$ and the set of coweights $\{H_\lambda:\lambda\in\Delta_V\}$ invariant. Here, by $\Delta_\g{k}$ we denote the set of roots $\alpha\in(\g{t}^\C)^*$ of the pair $(\g{k}^\C,\g{t}^\C)$ (note that any such root vanishes on the center $Z(\g{k}^\C)$), by $\Delta_V$ we denote the set of weights of the representation $\rho_*^\C$, and for any $\lambda\in\g{t}^*$ we define $H_\lambda\in\g{t}$ by $\langle H_\lambda,T\rangle=\lambda(T)$ for every $T\in\g{t}$, where $\langle\cdot,\cdot\rangle$ is an $\Aut(\g{k},\cal{F})$-invariant inner product on $\g{k}$. Moreover, since $\rho_*^\C$ is of real type, hence self-dual, $-\id_\g{t}$ leaves the set of coweights invariant.  Hence, $\Aut_\cal{F}^\pm(\Delta_\g{k},\Delta_V)\subset\Aut(\Delta_\g{k},\Delta_V)$.
\begin{fact}\label{fact:aut_Jt}
	If $T_1$, $T_2\in\cal{J}\cap\g{t}$, then $T_1\sim T_2$ if and only if $\varphi T_1=T_2$ for some $\varphi\in \Aut_\cal{F}^\pm(\Delta_\g{k},\Delta_V)$.
\end{fact}

Fix a set $\Pi_\g{k}=\{\alpha_1,\dots,\alpha_r\}$ of simple roots for the root system $\Delta_\g{k}$; here, $r$ is the rank of the semisimple part $[\g{k},\g{k}]$ of the compact Lie algebra $\g{k}=[\g{k},\g{k}]\oplus Z(\g{k})$. The subset $\bar{C}$ of $\g{t}$ defined by the inequalities $\alpha_i\geq 0$ for every $i=1,\dots, r$ constitutes a fundamental domain for the action of the Weyl group on $\g{t}$. It is the orthogonal product of $Z(\g{k})$ and a closed Weyl chamber in $\g{t}\cap[\g{k},\g{k}]$. We will denote
\begin{align*}
	\Out(\Delta_\g{k},\Delta_V)&=\{\varphi\in\Aut(\Delta_\g{k},\Delta_V):\varphi(\bar{C})=\bar{C}\},
	\\
	\Out^\pm_\cal{F}(\Delta_\g{k},\Delta_V)&=\{\varphi\in\Aut^\pm_\cal{F}(\Delta_\g{k},\Delta_V):\varphi(\bar{C})=\bar{C}\}.
\end{align*}
Thus, one obtains the following result, which is analogous to Fact~\ref{fact:aut_Jt}:
\begin{fact}\label{fact:J_Out}
		If $T_1$, $T_2\in\cal{J}\cap\bar{C}$, then $T_1\sim T_2$ if and only if $\varphi T_1=T_2$ for some $\varphi\in \Out_\cal{F}^\pm(\Delta_\g{k},\Delta_V)$.
\end{fact}
In~\cite{Do:tams}, the study of the action of $\Out_\cal{F}^\pm(\Delta_\g{k},\Delta_V)$ on $\cal{J}\cap\bar{C}$ was greatly facilitated by the introduction of a graph called lowest weight diagram. This graph generalizes the notion of extended Vogan diagram (which corresponds to the particular case when $\cal{F}$ is a homogeneous polar foliation induced by the isotropy representation $\rho$ of an inner symmetric space).

The \emph{lowest weight diagram} of $\cal{F}$ (or of $\rho_*^\C$) is constructed as follows. Start with the Dynkin diagram of $\Delta_\g{k}$, where each simple root in $\Delta_\g{k}$ is represented by a white node. Draw as many black nodes as there are lowest weights of $\rho_*^\C$, indicating the multiplicity. (We recall that a weight $\lambda\in\Delta_V$ is a lowest weight if $\lambda -\sum_{i=1}^r n_i \alpha_i\notin\Delta_V$, with $n_i\in\mathbb{Z}_{\geq 0}$, unless all $n_i$ vanish.)
Join each black node corresponding to a lowest weight $\lambda$ to the white nodes associated with simple roots $\alpha$ with  $\langle \alpha,\lambda\rangle\neq 0$ by means of a simple edge.
Finally, attach to each one of these new edges the integer value $2\langle\alpha,\lambda\rangle/\langle\alpha,\alpha\rangle$ as a label; if no label is attached, we understand that the associated integer is $-1$. The inner product on $\g{t}^*$ is induced from the $\Aut(\g{k},\cal{F})$-invariant one in $\g{k}$ in the natural way: $\langle\lambda,\mu\rangle=\langle H_\lambda,H_\mu\rangle$, for any $\lambda,\mu\in\g{t}^*$. An automorphism of a lowest weight diagram is a permutation of
its nodes preserving the black (and hence, white) nodes, the multiplicities of the associated weights, the edges, and the labels  of the edges between black nodes and white nodes.

\begin{fact}\label{fact:Out_lwd}
	Each element of $\Out_\cal{F}^\pm(\Delta_\g{k},\Delta_V)\subset \Out(\Delta_\g{k},\Delta_V)$ is an orthogonal linear transformation of $\g{t}$ that preserves the set of coroots and the set of coweights of $\rho_*^\C$. Therefore, it induces an automorphism of the lowest weight diagram of $\cal{F}$ in a natural way.
\end{fact}
Assume that one could prove that, conversely, each automorphism of the lowest weight diagram of $\cal{F}$ is induced by an element of $\Out_\cal{F}^\pm(\Delta_\g{k},\Delta_V)$. Then, by Fact~\ref{fact:J_Out},
the investigation of $\cal{J}\cap\bar{C}/\!\sim$ can be simplified by analyzing the symmetries of the lowest weight diagram and their induced action on $\cal{J}\cap\bar{C}$. This was precisely the approach followed in~\cite{Do:tams}, both for homogeneous polar foliations and FKM foliations with $m_1\leq m_2$.

\subsection{Analysis of the projections of $\cal{F}_{(8,7)}$ and $\cal{F}_{\underline{(8,7)}}$}\label{subsec:8,7}
In this subsection we will apply the theory explained in~\S\ref{subsec:approach} to the FKM foliations $\cal{F}_{(8,7)}$ and $\cal{F}_{\underline{(8,7)}}$. However, we start with the following known remark, which implies that, for these foliations, the numbers $N_\C$ and $N_\H$ are positive.

\begin{remark}\label{rem:quat_str_P}
	Any FKM foliation $\cal{F}$ satisfies $N_\C(\cal{F})\geq 1$; and also $N_\H(\cal{F})\geq 1$ provided that $m\geq 2$.  Indeed, using the notation in~Section~\ref{sec:aut}, since $m\geq 1$, then $P_0P_1$ is always a complex structure preserving $\cal{F}$; see~\cite[Theorem~6.1]{Do:tams}. Similarly, any FKM foliation $\cal{F}$ with $m\geq 2$ is preserved at least by the quaternionic structure $\g{q}=\spann\{P_0P_1, P_1P_2,P_0P_2\}$.
\end{remark}

In order to apply the theory in~\S\ref{subsec:approach} to $\cal{F}_{(8,7)}$ and $\cal{F}_{\underline{(8,7)}}$, we present below some information about the weights of the associated representations $\rho_*^\C$ and, especially, the proof of a converse to Fact~\ref{fact:Out_lwd}.

Let $V=\R^{32}$, and consider an FKM foliation  $\cal{F}\in\{\cal{F}_{(8,7)},\cal{F}_{\underline{(8,7)}}\}$ on $\sph(V)=\sph^{31}$. In this case, since the singular leaves of $\cal{F}$ have different dimensions, it follows that $\Aut(\cal{F})$ leaves each leaf of $\cal{F}$ invariant.
Hence, the maximal connected subgroup $\rho(K)$ of $\SO{V}$ that leaves each leaf of $\cal{F}$ invariant agrees with $\Aut^0(\cal{F})$, which, by Proposition~\ref{prop:2cases}, is in turn isomorphic to $\Spin{9}\cdot H$, where $H=\SO{2}$ or $H$ is trivial, depending on whether $\cal{F}=\cal{F}_{(8,7)}$ or  $\cal{F}=\cal{F}_{\underline{(8,7)}}$, respectively. Indeed, the representation $\rho\colon K\to \OG{V}$ introduced in~\S\ref{subsec:approach} is given by the tensor product $\R^{16}\otimes\R^2$ of the irreducible $\Spin{9}$-representation on $\R^{16}$ and the standard representation of $\SO{2}$ on $\R^2$ (if $H=\SO{2}$), and by two copies $\R^{16}\oplus\R^{16}$ of the irreducible $\Spin{9}$-representation on $\R^{16}$ (if $H$ is trivial).

We need to introduce some notation for the roots of $\g{k}$ and the weights of $\rho_*^\C$. We mostly follow the notation of~\cite[\S6]{Do:tams}. The maximal abelian subalgebra $\g{t}$ of $\g{k}$ splits as $\g{t}=\g{t}_s\oplus\g{t}_\g{h}$, where $\g{t}_s$ is a maximal abelian subalgebra of the Lie algebra $\g{so}(9)$ of $\Spin{9}$, and $\g{t}_\g{h}=\g{h}$ is one-dimensional or trivial, depending on the foliation. On the one hand, let $\{\alpha_1^s,\alpha_2^s,\alpha_3^s,\alpha_4^s\}$ be a system of simple roots for $\g{so}(9)$. It is well-known that there is an orthonormal basis $\{\omega_1^s,\omega_2^s,\omega_3^s,\omega_4^s\}$ of $\g{t}_s^*$ such that \begin{equation}\label{eq:alpha^s}
	\alpha_i^s=\omega_i^s-\omega_{i+1}^s\quad \text{for } i=1,2,3, \qquad \text{and}\qquad \alpha_4^s=\omega_4^s.
\end{equation}
The weights of the spin representation $\rho_s$ of $\g{so}(9)$ on $\C^{16}$ are $\frac{1}{2}(\pm\omega_1^s\pm\omega_2^s\pm\omega_3^s\pm\omega_4^s)$. On the other hand, if $\g{h}=\g{so}(2)$, the weights of the standard representation of $\g{so}(2)$ on $\C^2$ can be written as $\pm\omega$, for some non-zero $\omega\in\g{t}_\g{h}^*$. Therefore, the weights of $\rho_*^\C$ are:
\begin{equation}
\begin{aligned}\label{eq:weights}
 & \frac{1}{2}(\pm\omega_1^s\pm\omega_2^s\pm\omega_3^s\pm\omega_4^s)\pm\omega, &\text{if }&\cal{F}=\cal{F}_{(8,7)}\text{ and }\g{h}=\g{so}(2),
 \\
 &  \frac{1}{2}(\pm\omega_1^s\pm\omega_2^s\pm\omega_3^s\pm\omega_4^s), \text{ with multiplicity } 2, &\text{if }& \cal{F}= \cal{F}_{\underline{(8,7)}}\text{ and } \g{h}=0.
\end{aligned}
\end{equation}
In particular, $\rho_*^\C$ has two lowest weights  in the first case, namely  $\lambda^\pm=-\frac{1}{2}(\omega_1^s+\omega_2^s+\omega_3^s+\omega_4^s)\pm\omega$, and only one lowest weight (with multiplicity $2$)  in the second case, namely $\lambda=-\frac{1}{2}(\omega_1^s+\omega_2^s+\omega_3^s+\omega_4^s)$. Hence, the respective lowest weight diagrams adopt the following forms:

\begin{center}
\begin{tabular}{c@{\qquad \qquad}c}
\begin{tikzpicture}[node distance=1,g/.style={circle,inner sep=2.2,draw},lw/.style={circle,inner sep=2.2,fill=black, draw}]
	\node[g] (a1s)[label=below:$\alpha^s_1$] {};
	\node[g] (a2s) [right=of a1s, label=below:$\alpha^s_2$] {}
	edge [] (a1s);
	\node[g] (ap1) [right=of a2s, label=below:$\alpha^s_{3}$] {}
	edge [] (a2s);
	\node[g] (ap) [right=of ap1, label=below:$\alpha^s_4$] {};
	\draw ($(ap1)!.65!(ap)$) -- ($(ap1)!.65!(ap)+(-0.3,0.2)$);
	\draw ($(ap1)!.65!(ap)$) -- ($(ap1)!.65!(ap)+(-0.3,-0.2)$);
	\draw [] (ap1.north east) to (ap.north west);
	\draw [] (ap1.south east) to (ap.south west);
	\node[lw] (lw+) at ($(ap)+(30:1.2)$) [label=right:$\lambda^+$] {}
	edge [] (ap);
	\node[lw] (lw-) at ($(ap)+(-30:1.2)$) [label=right:$\lambda^-$] {}
	edge [] (ap);
\end{tikzpicture}
&
\begin{tikzpicture}[node distance=1,g/.style={circle,inner sep=2.2,draw},lw/.style={circle,inner sep=2.2,fill=black, draw}]
	\node[g] (a1s)[label=below:$\alpha^s_1$] {};
	\node[g] (a2s) [right=of a1s, label=below:$\alpha^s_2$] {}
	edge [] (a1s);
	\node[g] (ap1) [right=of a2s, label=below:$\alpha^s_{3}$] {}
	edge [] (a2s);
	\node[g] (ap) [right=of ap1, label=below:$\alpha^s_4$] {};
	\node[lw] (la) [right=of ap, label=below:$\lambda$,label=above:$(2)$] {}
	edge [] (ap);
	\draw ($(ap1)!.65!(ap)$) -- ($(ap1)!.65!(ap)+(-0.3,0.2)$);
	\draw ($(ap1)!.65!(ap)$) -- ($(ap1)!.65!(ap)+(-0.3,-0.2)$);
	\draw [] (ap1.north east) to (ap.north west);
	\draw [] (ap1.south east) to (ap.south west);
\end{tikzpicture}
\end{tabular}
\end{center}

The following result extends \cite[Theorem~6.3]{Do:tams} (which was only valid for FKM foliations with $m_1\leq m_2$) to the FKM foliations under consideration in this section.
\begin{proposition}\label{prop:out}
Consider an FKM foliation $\cal{F}\in\{\cal{F}_{(8,7)},\cal{F}_{\underline{(8,7)}}\}$ on $\sph(V)=\sph^{31}$. Then $\Out_\cal{F}^\pm(\Delta_\g{k},\Delta_V)$ is isomorphic to the group of automorphisms of the lowest weight diagram of $\cal{F}$.
The correspondence is the natural one given in Fact~\ref{fact:Out_lwd}.
\end{proposition}
\begin{proof}
In view of Fact~\ref{fact:Out_lwd}, we just have to show that each automorphism of the lowest weight diagram is induced by an element of $\Out_\cal{F}^\pm(\Delta_\g{k},\Delta_V)$. If $\cal{F}=\cal{F}_{\underline{(8,7)}}$, the diagram does not have any non-trivial automorphism, so there is nothing to prove. Hence, assume $\cal{F}=\cal{F}_{(8,7)}$. The only non-trivial automorphism of the diagram interchanges the nodes corresponding to the lowest weights $\lambda^+$ and $\lambda^-$, and fixes the remaining nodes (corresponding to the simple roots). If we denote by $\{e_1^s,e_2^s,e_3^s,e_4^s,e\}\subset\g{t}$ the dual basis of $\{\omega_1^s,\omega_2^s,\omega_3^s,\omega_4^s,\omega\}\subset\g{t}^*$, then the lowest coweights are $H_{\lambda^\pm}=-\frac{1}{2}
(e_1^s+e_2^s+e_3^s+e_4^s)\pm e$. Hence, the linear transformation $\varphi$ of $\g{t}=\g{t}_s\oplus\g{t}_\g{h}$ determined by the only non-trivial automorphism of the diagram is given by $\varphi\vert_{\g{t}_s}=\id_{\g{t}_s}$ and $\varphi\vert_{\g{t}_\g{h}}=-\id_{\g{t}_\g{h}}$. By Remark~\ref{rem:O(2)}, the matrix $\diag(-1,1)\in\OG{2}\subset \Aut(\cal{F}_{(8,7)})$ can be identified with an automorphism $A$ of $\cal{F}_{(8,7)}$. By the discussion just above Fact~\ref{fact:aut_k,F}, this automorphism determines an element $\Phi_A\in\Aut(\g{k},\cal{F})$ with $\Phi_A\vert_\g{k}=\Ad(A)$ and $\Phi_A\vert_V=A$. In particular, $\Phi_A\vert_{\g{t}_s}=\id_{\g{t}_s}$ and $\Phi_A\vert_{\g{t}_\g{h}}=-\id_{\g{t}_\g{h}}$. This implies, on the one hand,  $\Phi_A\vert_\g{t}=\varphi$, and on the other hand,   $\Phi_A\vert_\g{t}\in\Out_{\cal{F}}^\pm(\Delta_\g{k},\Delta_V)$.
\end{proof}

Now we have all the ingredients to complete the proof of Theorem~\ref{th:87}. According to the discussion in~\S\ref{subsec:approach} (see especially \eqref{eq:N_CH} and Fact~\ref{fact:equivalences}), we just have to calculate $N_\C(\cal{F})=|\cal{J}/\!\sim\!|=|\cal{J}\cap\bar{C}/\!\sim\!|$ and $N_\H(\cal{F})=|\cal{S}/\!\sim\!|=|\cal{S}_{\bar{C}}/\!\sim\!|$ for $\cal{F}\in\{\cal{F}_{(8,7)},\cal{F}_{\underline{(8,7)}}\}$. But we will actually compute explicit representatives of the moduli spaces $\cal{J}/\!\sim$ and $\cal{S}/\!\sim$.

\begin{proof}[Proof of Theorem~\ref{th:87}]
We will first consider the complex case. Denote by $\{e_1^s,e_2^s,e_3^s,e_4^s\}\subset\g{t}_s$ the dual basis of $\{\omega_1^s,\omega_2^s,\omega_3^s,\omega_4^s\}$, and, in the case $\cal{F}=\cal{F}_{(8,7)}$, let $e\in\g{t}_\g{h}$ be such that $\omega(e)=1$. An arbitrary element $T\in\g{t}$ can be written as $T=\sum_{i=1}^4x_i^s e_i^s + x e$, with $x_i^s,x\in\R$ (where the last addend can only appear if $\cal{F}=\cal{F}_{(8,7)}$, so we assume $x=0$ if $\cal{F}=\cal{F}_{\underline{(8,7)}}$). Then $T\in\bar{C}$ if and only if $\alpha_i^s(T)\geq 0$ for $i=1,2,3,4$, which by~\eqref{eq:alpha^s} is equivalent to $x_1^s\geq x_2^s\geq x_3^s\geq x_4^s\geq 0$; and by Fact~\ref{fact:pm1} and~\eqref{eq:weights}, $T\in\cal{J}$ if and only if $\frac{1}{2}(\pm x_1^s\pm x_2^s\pm x_3^s\pm x_4^s)\pm x\in\{\pm 1\}$ for all combinations of signs. Therefore,
\[
\cal{J}\cap\bar{C}=\{2e_1^s,\pm e\} \quad\text{ if } \cal{F}=\cal{F}_{(8,7)},\qquad \text{and}\qquad \cal{J}\cap\bar{C}=\{2e_1^s\} \quad\text{ if } \cal{F}=\cal{F}_{\underline{(8,7)}}.
\]
But in the first case, $\Out_{\cal{F}}^\pm(\Delta_\g{k},\Delta_V)=\{\id_\g{t},\varphi\}$, where $\varphi=\Phi_A$ is the transformation given in the proof of Proposition~\ref{prop:out}, which fixes $\g{t}_s$ pointwisely and interchanges $e$ and $-e$. Thus, thanks to Fact~\ref{fact:Out_lwd}, we conclude:
\[
\cal{J}\cap\bar{C}/\!\sim=\{[2e_1^s],[e]\} \quad\text{ if } \cal{F}=\cal{F}_{(8,7)},\qquad \text{and}\qquad \cal{J}\cap\bar{C}=\{[2e_1^s]\} \quad\text{ if } \cal{F}=\cal{F}_{\underline{(8,7)}},
\]
and hence $N_\C(\cal{F}_{(8,7)})=2$ and $N_\C(\cal{F}_{\underline{(8,7)}})=1$. This proves the claims of Theorem~\ref{th:87}  in the complex case.

Now, we have to deal with the quaternionic case. By Facts~\ref{fact:su_2_vs_complex} and~\ref{fact:injective}, in order to determine $\cal{S}_{\bar{C}} /\!\sim$ we have to decide which classes $[T]\in\cal{J}\cap\bar{C}/\sim$ have representatives that belong to $\g{su}_2$-subalgebras of $\g{k}\cong\g{so}(9)\oplus\g{h}$. On the one hand, $e$ does not belong to any $\g{su}(2)$-subalgebra of $\g{k}$, since $e$ spans the center of~$\g{k}$. 
On the other hand, the element $2e_1^s$ belongs to a subalgebra of~$\g{so}(9)$ isomorphic to $\g{so}(3)\cong\g{su}(2)$. Furthermore, it follows from Remark~\ref{rem:quat_str_P} that there exists at least
one quaternionic structure preserving~$\cal{F}$, which must therefore be equivalent to the one given by~$2e_1^s$. Thus $\cal{S}_{\bar{C}} /\!\sim$ consists of one element in both cases, namely the class of any $\g{su}(2)$-subalgebra of $\g{so}(9)$ containing $2e_1^s$. Therefore, $N_\H(\cal{F}_{(8,7)})=N_\H(\cal{F}_{\underline{(8,7)}})=1$.

Finally, the claim about the inhomogeneity of the projected foliations is a consequence of the inhomogeneity of $\cal{F}_{(8,7)}$ and $\cal{F}_{\underline{(8,7)}}$, cf.~\cite[Remark~7.1]{Do:tams} and~\cite[\S6]{DG:tohoku}.
\end{proof}

We will now prove Corollaries~\ref{cor:15,7} and~\ref{cor:codim1}. We start with the more general Corollary~\ref{cor:codim1}.

\begin{proof}[Proof of Corollary~\ref{cor:codim1}]	
	Let $\cal{G}$ be a codimension one polar foliation of $\K\P^n$. Then, there exists a codimension one polar foliation $\cal{F}$ of $\sph^{{a}(n+1)-1}$, ${a}=\dim_\R\K$, and a complex or quaternionic structure, $J$ or $\g{q}$, on $\R^{{a}(n+1)}$ preserving $\cal{F}$ such that $\cal{G}=\pi(\cal{F})$, where $\pi\colon \sph^{{a}(n+1)-1}\to\K\P^n$ is the Hopf fibration associated with $J$ or $\g{q}$.
	
	Assume first that $\cal{F}$ is reducible. This is equivalent to the fact that the regular leaves of $\cal{F}$ have $g\in\{1,2\}$ principal curvatures, and also equivalent to the extension $\hat{\cal{F}}$ of $\cal{F}$ via homotheties to $\R^{{a}(n+1)}$ being a product of two polar foliations of codimension one on lower dimensional Euclidean spaces, $\hat{\cal{F}}=\hat{\cal{F}}_1\times \hat{\cal{F}}_2$. Each $\hat{\cal{F}}_i$, $i=1,2$, is nothing but a foliation by concentric spheres centered at the origin, and hence its restriction $\cal{F}_i$ to the unit sphere has exactly one leaf, and thus $N_\K(\cal{F}_i)=1$ trivially. Then it follows from~\cite[Proposition~4.1]{Do:tams} and~\cite[Proposition~3.9]{DG:tohoku} that $N_\K(\cal{F})=1$, and hence  $\cal{G}=\pi(\cal{F})$ is congruent to one of the homogeneous reducible polar foliations $\cal{G}^\mathrm{red}_{\K\P^k,\K\P^{n-k-1}}$, $k\in\{0,\dots,n-1\}$.
	
	Now assume that $\cal{F}$ is irreducible. Then, the regular leaves of $\cal{F}$ have $g\in\{3,4,6\}$ principal curvatures. The case $g=3$ is impossible, since such foliations are induced by isotropy representations of non-inner symmetric spaces (see~\cite[end of~\S4.3]{Chi}), and then there do not exist complex (and hence quaternionic) structures preserving such foliations, according to~\cite[Theorem~5.1]{Do:tams}. If $g=6$, by the theory of isoparametric hypersurfaces in spheres, there are only two possible multiplicity pairs, namely $(m_1,m_2)\in\{(1,1),(2,2)\}$. The only isoparametric foliation with $g=6$ and multiplicities $(1,1)$ is homogeneous and induced by the isotropy representation of $\G/\SO{4}$ (this corresponds to the last row of Table~\ref{table:hom_c1}), whereas no isoparametric foliation with $g=6$ and multiplicities $(2,2)$ can be preserved by a complex (and hence, quaternionic) structure, as shown in~\cite[Proposition~2.3]{Do:tams}.
	
	We are left with the case $g=4$. By the classification results of isoparametric hypersurfaces with $g=4$ (see~\cite{Chi:jdg}, \cite{Chi}), we know that $\cal{F}$ is either an inhomogeneous FKM foliation or a homogeneous foliation induced by the isotropy representation of a rank two irreducible symmetric space. Up to congruence, the inhomogeneous FKM foliations satisfy the inequality $m_1\leq m_2$ (according to the notation in Section~\ref{sec:aut}), except the two foliations $\cal{F}_{(8,7)}$ and $\cal{F}_{\underline{(8,7)}}$ with $(m_1,m_2)=(8,7)$. The values of $N_\K$ specified in Table~\ref{table:FKM} have been extracted from~\cite[Main Theorem~2]{Do:tams}  and \cite[Theorem~1.2]{DG:tohoku} for the cases with $m_1\leq m_2$, and from Theorem~\ref{th:87} for $(m_1,m_2)=(8,7)$. Finally, Table~\ref{table:hom_c1} specifies the homogeneous foliations $\cal{F}$ with $N_\C(\cal{F})\geq 1$ in terms of the associated rank two irreducible symmetric spaces $G/K$. The values of $N_\C(\cal{F})$ and $N_\H(\cal{F})$ have been transcribed from~\cite[Main Theorem~2, or~\S5.3]{Do:tams} and~\cite[Theorem~1.1]{DG:tohoku}. From these results, the numbers $N_\C^\mathrm{h}(\cal{F})$ and $N_\H^\mathrm{h}(\cal{F})$ of homogeneous examples obtained as projections of $\cal{F}$ can also be easily derived, using the following rule:  $N_\C^\mathrm{h}(\cal{F})=1$ if and only if the associated $G/K$ is Hermitian, and $N_\C^\mathrm{h}(\cal{F})=0$ otherwise; and $N_\H^\mathrm{h}(\cal{F})=1$ if and only if the associated $G/K$ is quaternionic-K\"ahler, and $N_\H^\mathrm{h}(\cal{F})=0$ otherwise. Note that some of these homogeneous foliations $\cal{F}$ are also of FKM type. This information included in Table~\ref{table:hom_c1} can be derived from~\cite[\S4.4]{FKM:mathz} along with~\cite[end of~\S4.3]{Chi}.
\end{proof}

\begin{proof}[Proof of Corollary~\ref{cor:15,7}]
	This result can be obtained directly from Corollary~\ref{cor:codim1} by inspecting which foliations $\cal{F}$ in Tables~\ref{table:FKM} and~\ref{table:hom_c1} live in a $31$-dimensional sphere. It turns out that all such foliations are of FKM type.  Therefore, Table~\ref{table:15,7} is obtained by listing all (congruence classes of) FKM foliations $\cal{F}$ with multiplicities $(m_1,m_2)$ with $2m_1+2m_2+1=31$ (which can be done by consulting~\cite[Proposition~5.1]{Chi} or~\cite[\S4.3]{FKM:mathz}),  informing their homogeneity, and specifying the values $N_\K(\cal{F})$ and $N_\K^\mathrm{h}(\cal{F})$ with the help of Tables~\ref{table:FKM} and~\ref{table:hom_c1}.
\end{proof}

\end{document}